\pdfoutput=1
\documentclass{shinyart}

\usepackage[utf8]{inputenc}

\usepackage{shinybib}

\usepackage{enumitem} 
\usepackage{booktabs}
\usepackage[%
scientific-notation=true,
exponent-product=\cdot,
tight-spacing=true,
table-column-width=1.3cm,
table-format=1.2e1,
]{siunitx}
\usepackage{autonum}

\newcommand{\eps}{\varepsilon}
\renewcommand{\phi}{\varphi}
\newcommand{\N}{\mathbb{N}}
\newcommand{\R}{\mathbb{R}}
\newcommand{\Rbar}{\overline{\mathbb{R}}}

\newcommand{\calG}{\mathcal{G}}
\newcommand{\dual}[1]{\langle #1 \rangle}

\newcommand{\norm}[1]{\| #1 \|}
\newcommand{\set}[2]{\left\{#1:#2\right\}}
\newcommand{\wkto}{\rightharpoonup}

\DeclareMathOperator{\dom}{\mathrm{dom}}
\DeclareMathOperator{\ran}{\mathrm{ran}}
\DeclareMathOperator{\Id}{\mathrm{Id}}
\DeclareMathOperator{\co}{\mathrm{co}}
\usepackage{graphicx}
\usepackage{subcaption}
\usepackage{pgfplots}
\pgfplotsset{compat=newest}
\pgfplotsset{plot coordinates/math parser=false}

\title{Convex regularization of discrete-valued inverse problems}
\author{Christian Clason\thanks{%
        University of Duisburg-Essen, Faculty of Mathematics,
        Thea-Leymann-Str.~9,
        45127 Essen, Germany
    (\email{christian.clason@uni-due.de}, \email{tram.do@uni-due.de})}
    \and 
    Thi Bich Tram Do\footnotemark[1]
}
\date{July 26, 2017}

\hypersetup{
    pdftitle={Convex regularization of discrete-valued inverse problems},
    pdfauthor={Christian Clason, Thi Bich Tram Do},
    pdfkeywords={inverse problem, regularization, convex, discrete, multibang}
}
\begin{document}

\maketitle

\begin{abstract}
    This work is concerned with linear inverse problems where a distributed parameter is known \emph{a priori} to only take on values from a given discrete set. This property can be promoted in Tikhonov regularization with the aid of a suitable convex but nondifferentiable regularization term. This allows applying standard approaches to show well-posedness and convergence rates in Bregman distance. Using the specific properties of the regularization term, it can be shown that convergence (albeit without rates) actually holds pointwise. Furthermore, the resulting Tikhonov functional can be minimized efficiently using a semi-smooth Newton method. Numerical examples illustrate the properties of the regularization term and the numerical solution.
\end{abstract}

\section{Introduction}
\label{sec:introduction}

We consider Tikhonov regularization of inverse problems, where the unknown parameter to be reconstructed is a distributed function that only takes on values from a given discrete set (i.e., the values are known, but not in which points they are attained). Such problems can occur, e.g., in nondestructive testing or medical imaging; a similar task also arises as a sub-step in segmentation or labelling problems in image processing. The question we wish to address here is the following: If such strong \emph{a priori} knowledge is available, how can it be incorporated in an efficient manner? Specifically, if $X$ and $Y$ are function spaces, $F:X\to Y$ denotes the parameter-to-observation mapping, and $y^\delta \in Y$ is the given noisy data, we would wish to solve the constrained Tikhonov functional
\begin{equation}\label{eq:intro:nonconvex}
    \min_{u\in U} \frac12\norm{F(u) - y^\delta}_Y 
\end{equation}
for
\begin{equation}\label{eq:intro:nonconvex_constraint}
    U := \set{u\in X}{u \in \{u_1,\dots,u_d\} \text{ pointwise}},
\end{equation}
where $u_1,\dots,u_d\in\R$ are the known parameter values. However, this set is nonconvex, and hence the functional in \eqref{eq:intro:nonconvex} is not weakly lower-semicontinuous and can therefore not be treated by standard techniques. (In particular, it will in general not admit a minimizer.) A common strategy to deal with such problems is by convex relaxation, i.e., replacing $U$ by its convex hull
\begin{equation}
    \co U = \set{u\in X}{u \in [u_1,u_d] \text{ pointwise}}.
\end{equation}
This turns \eqref{eq:intro:nonconvex} into a classical \emph{bang-bang} problem, whose solution is known to generically take on only the values $u_1$ or $u_d$; see, e.g., \cite{Troeltzsch:1979,Troeltzsch:1996}.
If $d>2$, intermediate parameter values are therefore lost in the reconstruction. (Here we would like to remark that a practical regularization should not only converge as the noise level tends to zero but also yield informative reconstructions for fixed -- and ideally, a large range of -- noise levels.) 
As a remedy, we propose to add a convex regularization term that promotes reconstructions in $U$ (rather than merely in $\co U$) for the convex relaxation. Specifically, we choose the convex integral functional
\begin{equation}
    \calG:X\to\R,\qquad \calG(u) := \int g(u(x))\,dx,
\end{equation}
for a convex integrand $g:\R\to\R$ with a polyhedral epigraph whose vertices correspond to the known parameter values $u_1,\dots,u_d$. Just as in $L^1$ regularization for sparsity (and in linear optimization), it can be expected that minimizers are found at the vertices, thus yielding the desired structure. 

This approach was first introduced in \cite{CK:2013} in the context of linear optimal control problems for partial differential equations, where the so-called \emph{multi-bang} (as a generalization of bang-bang) penalty $\calG$ was obtained as the convex envelope of a (nonconvex) $L^0$ penalization of the constraint $u\in U$. The application to nonlinear control problems and the limit as the $L^0$ penalty parameter tends to infinity were considered in \cite{CK:2015}, and our particular choice of $\calG$ is based on this work. The extension of this approach to vector-valued control problems was carried out in \cite{CTW:2016}.

Our goal here is therefore to investigate the use of the multi-bang penalty from \cite{CK:2015} as a regularization term in inverse problems, in particular addressing convergence and convergence rates as the noise level and the regularization parameter tend to zero. Due to the convexity of the penalty, these follow from standard results on convex regularization if convergence is considered with respect to the Bregman distance. The main contribution of this work is to show that due to the structure of the pointwise penalty, this convergence can be shown to actually hold pointwise. Since the focus of our work is the novel convex regularization term, we restrict ourselves to linear problems for the sake of presentation. However, all results carry over in a straightforward fashion to nonlinear problems. Finally, we describe following \cite{CK:2013,CK:2015} the computation of Tikhonov minimizers using a path-following semismooth Newton method.

\bigskip

Let us briefly mention other related literature. Regularization with convex nonsmooth functionals is now a widely studied problem, and we only refer to the monographs \cite{Scherzer:2009,HKKS,ItoJin} as well as the seminal works \cite{burger2004convergence,Poeschl:2007,Resmerita:2008,Flemming:2011}.
To the best of our knowledge, this is the first work treating regularization of general inverse problems with discrete-valued distributed parameters. As mentioned above, similar problems occur frequently in image segmentation or, more generally, image labelling problems. The former are usually treated by (multi-phase) level set methods \cite{Vese2002} or by a combination of total variation minimization and thresholding \cite{Cai:2013}. More general approaches to image labelling problems are based on graph-cut algorithms \cite{Ishikawa:2003,Bae2009} or, more recently, vector-valued convex relaxation \cite{Goldluecke2010,Lellmann:2011}. Both multi-phase level sets and vector-valued relaxations, however, have the disadvantage that the dimension of the parameter space grows quickly with the number of admissible values, which is not the case in our approach. On the other hand, our approach assumes, similar to \cite{Ishikawa:2003}, a linear ordering of the desired values which is not necessary in the vector-valued case; see also \cite{CTW:2016}.

\bigskip

This work is organized as follows. In \cref{sec:multibang}, we give the concrete form of the pointwise multi-bang penalty $g$ and summarize its relevant properties. \Cref{sec:convergence} is concerned with well-posedness, convergence, and convergence rates of the corresponding Tikhonov regularization. Our main result, the pointwise convergence of the regularized solutions to the true parameter, is the subject of \cref{sec:pointwise}. We also briefly discuss the structure of minimizers for given $y^\delta$ and fixed $\alpha>0$ in \cref{sec:structure}. Finally, we address the numerical solution of the Tikhonov minimization problem using a semismooth Newton method in \cref{sec:semismooth} and apply this approach to an inverse source problem for a Poisson equation in \cref{sec:examples}.

\section{Multi-bang penalty}\label{sec:multibang}

Let $u_1<\cdots<u_d\in\R$, $d\geq 2$, be the given admissible parameter values and $\Omega\subset \R^n$, $n\in\N$, be a bounded domain. Following \cite[\S\,3]{CK:2015}, we define the corresponding multi-bang penalty
\begin{equation}
    \calG:L^2(\Omega)\to\Rbar,\qquad \calG(u) = \int_\Omega g(u(x))\,dx,
\end{equation}
for $g:\R\to\Rbar$ defined by
\begin{equation}
    g(v) =  
    \begin{cases}
        \frac12 \left((u_{i}+u_{i+1})v - u_iu_{i+1}\right) & \text{if }v \in [u_i,u_{i+1}], \quad1\leq i < d,\\
        \infty & \text{else}.
    \end{cases}
\end{equation}
(Note that we have now included the convex constraint $u\in \co U$ in the definition of $\calG$.) This choice can be motivated as the convex hull of $\frac12\norm{\cdot}_{L^2(\Omega)}^2 + \delta_U$, where $\delta_U$ denotes the indicator function of the set $U$ defined in \eqref{eq:intro:nonconvex_constraint} in the sense of convex analysis, i.e., $\delta_U(u) = 0$ if $u\in U$ and $\infty$ else; see \cite[\S\,3]{CK:2015}. Setting 
\begin{equation}
    g_i(v):= \frac12 \left((u_{i}+u_{i+1})v - u_iu_{i+1}\right),\qquad 1\leq i<d,
\end{equation}
it is straightforward to verify that
\begin{equation}
    g(v) = \max_{1\leq i <d} g_i(v),\qquad v\in [u_1,u_d],
\end{equation}
and hence $g$ is the pointwise supremum of affine functions and therefore convex and continuous on the interior of its effective domain $\dom g = [u_1,u_d]$.

We can thus apply the sum rule and maximum rule of convex analysis (see, e.g., \cite[Props.~4.5.1 and 4.5.2, respectively]{Schirotzek:2007}), and obtain for the convex subdifferential at $v\in \dom g$ that
\begin{equation}
    \begin{aligned}
        \partial g(v) &= \partial\left(\max_{1\leq i<d}g_i + \delta_{[u_1,u_d]}\right)(v)\\
                      &= \partial \left(\max_{1\leq i< d} g_i\right)(v) + \partial\delta_{[u_1,u_d]}(v) \\
                      &= 
        \co \left(\bigcup_{i:g(v)=g_i(v)}g'_i(v)\right) + \partial\delta_{[u_1,u_d]}(v).
    \end{aligned}
\end{equation}
Using the definition of $g_i$ together with the classical characterization of the subdifferential of an indicator function via its normal cone yields the explicit characterization
\begin{equation}\label{eq:multibang:dg}
    \partial g(v) =  \begin{cases}
        \left(-\infty,\tfrac12 (u_1+u_2)\right] & \text{if }v = u_1,\\
        \left\{\tfrac12(u_i+u_{i+1})\right\} & \text{if }v\in (u_i,u_{i+1}),\quad 1\leq i < d,\\
        \left[\tfrac12(u_{i-1}+u_{i}),\tfrac12(u_{i}+u_{i+1})\right] & \text{if }v = u_{i},\qquad\qquad\!\!\! 1< i<d,\\
        \left[\tfrac12(u_{d-1}+u_d),\infty\right) & \text{if }v = u_d,\\
        \emptyset &\text{else}.
    \end{cases}
\end{equation}

In \cref{sec:structure,sec:semismooth}, we will also make use of the subdifferential of the Fenchel conjugate $g^*$ of $g$. Here we can use the fact that $g$ is convex and hence $q\in \partial g(v)$ if and only if $v\in \partial g^*(q)$ (see, e.g., \cite[Prop.~4.4.4]{Schirotzek:2007}) to obtain 
\begin{equation}\label{eq:multibang:dgstar}
    \partial g^*(q) \in \begin{cases}
        \{u_1\} & \text{if }q \in \left(-\infty,\tfrac12(u_1+u_2)\right),\\
        [u_i,u_{i+1}] &\text{if } q = \tfrac12(u_{i}+u_{i+1}),\qquad\qquad\qquad\quad 1\leq i<d,\\
        \{u_i\} &\text{if } q \in \left(\tfrac12(u_{i-1}+u_{i}),\tfrac12(u_{i}+u_{i+1})\right), \quad 1< i<d,\\
        \{u_d\} &\text{if } q \in \left(\tfrac12(u_{d-1}+u_d),\infty\right),\\
        \emptyset &\text{else.}
    \end{cases}
\end{equation}
(Note that subdifferentials are always closed.)
We illustrate these characterizations for a simple example in \cref{fig:multibang}.
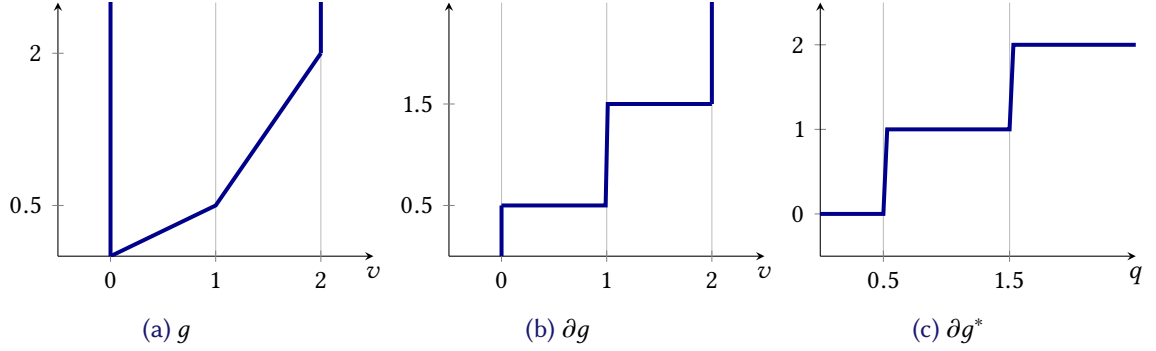
\begin{figure}[t]
    \begin{subfigure}[b]{0.3\textwidth}
        \centering
        \begin{tikzpicture}[baseline,style={font=\small}]
            \begin{axis}[%
                width=1.3\textwidth,
                xmin=-0.5,
                xmax=2.5,
                ymax=2.5,
                xlabel style={at={(axis cs:2.5,0)}},
                xlabel={$v$},
                xtick=\empty,
                extra x tick style={grid=major},
                extra x ticks={0,1,2},
                extra x tick labels={$0$,$1$,$2$},
                ytick=\empty,
                extra y ticks={0.5,2},
                extra y tick labels={$0.5$,$2$},
                axis y line=left,
                axis x line=bottom,
                ]
                \addplot[%
                domain=0:2,
                color=DarkBlue,solid,line width=1.5pt,
                ]
                {
                    and(x>0,x<=1)*(0.5*((0+1)*x-0*1))+%
                    and(x>1,x<=2)*(0.5*((1+2)*x-1*2))%
                };
                \addplot[color=DarkBlue,solid,line width=1.5pt] coordinates {(0,0) (0,3)};
                \addplot[color=DarkBlue,solid,line width=1.5pt] coordinates {(2,2) (2,3)};
            \end{axis}
        \end{tikzpicture}  
        \caption{$g$}\label{fig:multibang:g}
    \end{subfigure}
    \hfill
    \begin{subfigure}[b]{0.3\linewidth}
        \centering
        \begin{tikzpicture}[baseline,style={font=\small}]
            \begin{axis}[%
                width=1.3\linewidth,
                xmin=-0.5,
                xmax=2.5,
                ymax=2.5,
                xlabel style={at={(axis cs:2.5,0)}},
                xlabel={$v$},
                xtick=\empty,
                extra x tick style={grid=major},
                extra x ticks={0,1,2},
                extra x tick labels={$0$,$1$,$2$},
                ytick=\empty,
                extra y ticks={0.5,1.5},
                extra y tick labels={$0.5$,$1.5$},
                axis y line=left,
                axis x line=bottom,
                ]
                \addplot[%
                samples=100,
                domain=0:2,
                color=DarkBlue,solid,line width=1.5pt,
                ]
                {
                    and(x>=0,x<1)*(0.5*(0+1))+%
                    and(x>=1,x<2)*(0.5*(1+2))%
                };
                \addplot[color=DarkBlue,solid,line width=1.5pt] coordinates {(0,0) (0,0.5)};
                \addplot[color=DarkBlue,solid,line width=1.5pt] coordinates {(2,1.5) (2,3)};
            \end{axis}
        \end{tikzpicture}
        \caption{$\partial g$}\label{fig:multibang:dg}
    \end{subfigure}
    \hfill
    \begin{subfigure}[b]{0.3\linewidth}
        \centering
        \begin{tikzpicture}[baseline,style={font=\small}]
            \begin{axis}[%
                width=1.3\linewidth,
                xmin=0,
                xmax=2.5,
                ymin=-0.5,
                ymax=2.5,
                xlabel style={at={(axis cs:2.5,-0.5)}},
                xlabel={$q$},
                xtick=\empty,
                extra x tick style={grid=major},
                extra x ticks={0.5,1.5},
                extra x tick labels={$0.5$,$1.5$},
                axis y line=left,
                axis x line=bottom,
                ]
                \addplot[%
                samples=100,
                domain=-0.5:2.5,
                color=DarkBlue,solid,line width=1.5pt,
                ]
                {
                    (x<=0.5)*(0)+%
                    and(x>0.5,x<=1.5)*(1)+%
                    (x>1.5)*(2)%
                };
            \end{axis}
        \end{tikzpicture}   
        \caption{$\partial g^*$}\label{fig:multibang:dgstar}
    \end{subfigure}

    \caption{Structure of pointwise multibang penalty for the choice $(u_1,u_2,u_3)=(0,1,2)$}
    \label{fig:multibang}
\end{figure}

Finally, since $g$ is proper, convex, and lower semi-continuous by construction, the corresponding integral functional $\calG:L^2(\Omega)\to\Rbar$ is proper, convex and weakly lower semicontinous as well; see, e.g., \cite[Proposition 2.53]{Barbu}. Furthermore, the sub\-differential can be computed pointwise as
\begin{equation}\label{eq:multibang:dG}
    \partial\calG(u) = \set{v\in L^2(\Omega)}{v(x) \in \partial g(u(x))\quad\text{for almost every }x\in\Omega},
\end{equation}
see, e.g., \cite[Prop.~2.53]{Barbu}. The same is true for the Fenchel conjugate $\calG^*:L^2(\Omega)\to\Rbar$ and hence for $\partial\calG^*$ (which is thus an element of $L^\infty(\Omega)$ instead of $L^2(\Omega)$); see, e.g., \cite[Props.~IV.1.2, IX.2.1]{Ekeland:1999a}.

\section{Multi-bang regularization}\label{sec:convergence}

We consider for a linear operator $K:X\to Y$ between the Hilbert spaces $X=L^2(\Omega)$ and $Y$ and exact data $y^\dag \in Y$ the inverse problem of finding $u\in X$ such that
\begin{equation}\label{eq:inverse_prob}
    Ku=y^\dag.
\end{equation}
We assume that $K$ is weakly closed, i.e., $u_n\wkto u$ and $Ku_n\wkto y$ imply $y=Ku$.
For the sake of presentation, we also assume that \eqref{eq:inverse_prob} admits a solution $u^\dag\in X$. Let now $y^\delta \in Y$ be given noisy data with $\norm{y^\delta -y^\dag}_Y\leq \delta$ for some noise level $\delta >0$. The \emph{multi-bang regularization} of \eqref{eq:inverse_prob} for $\alpha>0$ then consists in solving
\begin{equation}\label{eq:tikhonov}
    \min_{u\in X} \frac12\norm{Ku-y^\delta}_Y^2 + \alpha \calG(u).
\end{equation}
Since $\calG$ is proper, convex and semi-continuous with bounded effective domain $\co U$, and $K$ is weakly closed, the following results can be proved by standard semi-continuity methods; see also \cite{CK:2015,CTW:2016}.
\begin{proposition}[Existence and uniqueness]
    For every $\alpha >0$, there exists a minimizer $u_\alpha^\delta$ to  \eqref{eq:tikhonov}. If $K$ is injective, this minimizer is unique.
\end{proposition}
\begin{proposition}[Stability]
    Let $\{y_n\}_{n\in\N}\subset Y$ be a sequence converging strongly to $y^\delta\in Y$ and $\alpha >0$ be fixed. Then the corresponding sequence of minimizers $\{u_n\}_{n\in\N}$ to \eqref{eq:tikhonov}  contains a subsequence converging weakly to a minimizer $u^\delta_\alpha$.
\end{proposition}

We now address convergence for $\delta\to0$.
Recall that an element ${u^\dag}\in X$ is called a $\calG$-minimizing solution to \eqref{eq:inverse_prob} if it is a solution to \eqref{eq:inverse_prob} and 
$\calG(u^\dag)\le \calG(u)$ for all solutions $u$ to \eqref{eq:inverse_prob}.
The following result is standard as well; see, e.g., \cite{Scherzer:2009,HKKS,ItoJin}.
\begin{proposition}[Convergence]\label{thm:convergence}
    Let $\{y^{\delta_n}\}_{n\in\N}\subset Y$ be a sequence of noisy data with $\norm{y^{\delta_n}-y^\dag}_Y\leq \delta_n \to 0$, and choose $\alpha_n:=\alpha_n(\delta_n)$ satisfying
    \begin{equation}\label{eq:ass_alpha}
        \lim_{n\to\infty}\frac{\delta_n^2}{\alpha_n}=0 \qquad \text{and}\qquad \lim_{n\to\infty}\alpha_n=0.
    \end{equation}
    Then the corresponding sequence of minimizers  $\{u_{\alpha_n}^{\delta_n}\}_{n\in\N}$ to \eqref{eq:tikhonov} contains a subsequence converging weakly to a $\calG$-minimizing solution $u^\dag$.
\end{proposition}

For convex nonsmooth regularization terms, convergence rates are usually derived in terms of the Bregman distance \cite{bregman1967relaxation}, which is defined for $u_1,u_2\in X$ and $p_1\in\partial\calG(u_1)$ as
\begin{equation}\label{eq:bregman}
    d_{\calG}^{p_1}(u_2,u_1)=\calG(u_2)-\calG(u_1)-\dual{p_1,u_2-u_1}_X.
\end{equation}
From the convexity of $\calG$, it follows that $d_{\calG}^{p_1}(u_2,u_1)\geq 0$ for all $u_2\in X$. 
Furthermore, we have from, e.g., \cite[Lem.~3.8]{ItoJin} the so-called \emph{three-point identity}
\begin{equation}\label{eq:threepoint} 
    d_{\calG}^{p_1}(u_3,u_1)=d_{\calG}^{p_2}(u_3,u_2)+d_{\calG}^{p_1}(u_2,u_1)+(p_2-p_1)(u_3-u_2)
\end{equation}
for any $u_1,u_2,u_3\in X$ and $p_1\in\calG(u_1)$ and $p_2\in\partial\calG(u_2)$.
Finally, we point out that due to the pointwise characterization \eqref{eq:multibang:dG} of the subdifferential of the integral functional $\calG$, we have that
\begin{equation}\label{eq:bregman_pointwise}
    d_{\calG}^{p}(u_2,u_1) = \int_\Omega d_g^{p(x)}(u_2(x),u_1(x))dx
\end{equation}
for
\begin{equation}
    d_g^q(v_2,v_1) = g(v_2)-g(v_1) - q(v_2-v_1).
\end{equation}

Standard arguments can then be used to show convergence rates for \emph{a priori} and \emph{a posteriori} parameter choice rules under the usual source conditions; see, e.g.,
\cite{burger2004convergence,Resmerita:2008,Scherzer:2009,HKKS,ItoJin}. Here we follow the latter and assume that there exists a $w\in Y$ such that
\begin{equation}\label{eq:source}
    p^\dag:=K^*w\in\partial \calG(u^\dag).
\end{equation}
Under the \emph{a priori} choice rule
\begin{equation}\label{eq:apriori}
    \alpha = c \delta\qquad \text{for some }c>0,
\end{equation}
we obtain the following convergence rate from, e.g., \cite[Cor.~3.4]{ItoJin}.
\begin{proposition}[Convergence rate, \emph{a priori}]
    \label{thm:conv_apriori}
    Assume that the source condition \eqref{eq:source} holds and that $\alpha=\alpha(\delta)$ is chosen according to \eqref{eq:apriori}. Then there exists a $C>0$ such that 
    \begin{equation}
        d_\calG^{p^\dag}(u_{\alpha}^\delta,u^\dag) \leq C\delta.
    \end{equation}
\end{proposition}
We obtain the same rate under the classical Morozov discrepancy principle
\begin{equation}\label{eq:morozov}
    \delta<\norm{Ku_\alpha^\delta-y^\delta}_Y\le \tau\delta,
\end{equation}
for some $\tau>1$ from, e.g., \cite[Thm.~3.15]{ItoJin}.
\begin{proposition}[Convergence rate, \emph{a posteriori}]
    \label{thm:conv_aposteriori}
    Assume that the source condition \eqref{eq:source} holds and  that $\alpha=\alpha(\delta)$ is chosen according to \eqref{eq:morozov}. Then there exists a $C>0$ such that 
    \begin{equation}
        d_\calG^{p^\dag}(u_{\alpha}^\delta,u^\dag) \leq C\delta.
    \end{equation}
\end{proposition}

\section{Pointwise convergence}\label{sec:pointwise}

The pointwise definition \eqref{eq:bregman_pointwise} of the Bregman distance together with the explicit pointwise characterization \eqref{eq:multibang:dg} of subgradients allows us to show that the convergence in \cref{thm:convergence} is actually pointwise if $u^\dag(x)\in \{u_1,\dots,u_d\}$ almost everywhere. 
The following lemma provides the central argument for pointwise convergence.
\begin{lemma}\label{lem:conv_pointwise}
    Let $v^\dag\in\{u_1,\dots,u_d\}$ and $q^\dag\in\partial g(v^\dag)$ satisfying
    \begin{equation}\label{eq:ass_p}
        q^\dag\in
        \begin{cases}
            \left\{\tfrac{1}{2}(u_i+u_{i+1})\right\} &\text{if } v^\dag\in(u_i,u_{i+1}),\quad 1\le i<d,\\
            \left(\tfrac{1}{2}(u_i+u_{i-1}),\tfrac{1}{2}(u_i+u_{i+1})\right), &\text{if }v^\dag=u_i,\qquad\qquad\! 1<i<d\\
            \left(-\infty,\tfrac{1}{2}(u_1+u_2)\right), &\text{if }v^\dag=u_1\\
            \left(\tfrac{1}{2}(u_d+u_{d-1}),\infty\right), &\text{if }v^\dag=u_d
        \end{cases}
    \end{equation}

    Furthermore, let $\{v_n\}_{n\in\N}\subset [u_1,u_d]$ be a sequence with 
    \begin{equation}\label{eq:conv_pointwise}
        d_g^{q^\dag}(v_n,v^\dag)\rightarrow 0.
    \end{equation}
    Then, $v_n \to v^\dag$.
\end{lemma}
\begin{proof}
    We argue by contraposition: Assume that $v_n$ does not converge to $v^\dag=u_i$ for some $1\leq i \leq d$. Then there exists an $\eps>0$ such that for every $n_0\in\N$, there is an $n\ge n_0$ with $|v_n-v^\dag|>\eps$, i.e., either $v_n>u_i+\eps$ or $v_n<u_i-\eps$. 
    We now further discriminate these two cases. (Note that some cases cannot occur if $i=1$ or $i=d$.)
    \begin{enumerate}[label=(\roman*)]
        \item $v_n>u_{i+1}$: Then, $v_n\in(u_k,u_{k+1}]$ for some $k\ge i+1$. The three point identity \eqref{eq:threepoint} yields that
            \begin{equation} 
                d_g^{q^\dag}(v_n,v^\dag)=d_g^{q_{i+1}}(v_n,u_{i+1})+d_g^{q^\dag}(u_{i+1},v^\dag)+(q_{i+1}-q^\dag)(v_n-u_{i+1})
            \end{equation}
            for $q_{i+1}\in\partial g(u_{i+1})$. We now estimate each term separately. The first term is nonnegative by the properties of Bregman distances. For the last term, we can use the assumption \eqref{eq:ass_p} and the pointwise characterization \eqref{eq:multibang:dg} to obtain
            \begin{equation}
                q^\dag\in\left(\tfrac{1}{2}(u_i+u_{i-1}),\tfrac{1}{2}(u_i+u_{i+1})\right)\quad\text{and}\quad q_{i+1}\in\left[\tfrac{1}{2}(u_{i+1}+u_i),\tfrac{1}{2}(u_{i+1}+u_{i+2})\right],
            \end{equation}
            which implies that $q_{i+1}-q^\dag>0$. By assumption we have $v_n-u_{i+1}>0$, which together implies that the last term is strictly positive. For the second term, we can use that $v^\dag,u_{i+1}\in[u_i,u_{i+1}]$ to simplify the Bregman distance to
            \begin{equation}
                d_g^{q^\dag}(u_{i+1},v^\dag)=\frac{1}{2}(u_{i+1}-u_i)(u_{i+1}+u_i-2q^\dag) >0,
            \end{equation}
            again by assumption \eqref{eq:ass_p}. Since this term is independent of $n$, we obtain the estimate
            \begin{equation}
                d_g^{q^\dag}(v_n,v^\dag)>d_g^{q^\dag}(u_{i+1},v^\dag)=:\eps_1>0.
            \end{equation} 
        \item $u_i<v_n\le u_{i+1}$: In this case, we can again simplify 
            \begin{equation}
                d_g^{q^\dag}(v_n,v^\dag)=\frac{1}{2}(u_{i+1}+u_i-2q^\dag)(v_n-v^\dag)>C_1\eps,
            \end{equation}
            since $C_1:=\frac{1}{2}(u_{i+1}+u_i-2q^\dag)>0$ by assumption \eqref{eq:ass_p} and $v_n-v^\dag>\eps$ by hypothesis. 
        \item $v_n<u_i$: We argue similarly to either obtain
            \begin{align}
                d_g^{q^\dag}(v_n,v^\dag)&>d_g^{q^\dag}(u_{i-1},v^\dag)=:\eps_2>0
                \intertext{or}
                d_g^{q^\dag}(v_n,v^\dag)&>C_2\eps                
            \end{align}
            for $C_2:=-\frac{1}{2}(u_{i-1}+u_i-2q^\dag)>0$.
    \end{enumerate}
    Thus if we set $\tilde \eps:=\min\{\eps_1,\eps_2,C_1\eps,C_2\eps\}$, for every $n_0\in\N$ we can find $n\geq n_0$ such that $d_g^{q^\dag}(v_n,v^\dag)>\tilde \eps>0$. Hence, $d_g^{q^\dag}(v_n,v^\dag)$ cannot converge to $0$.
\end{proof}

Assumption \eqref{eq:ass_p} can be interpreted as a strict complementarity condition for $q^\dag$ and $v^\dag$. Comparing \eqref{eq:ass_p} to \eqref{eq:multibang:dg}, we point out that such a choice of $q^\dag$ is always possible.
If $v^\dag\notin\{u_1,\dots,u_d\}$, on the other hand, convergence in Bregman distance is uninformative.
\begin{lemma}
    Let $v^\dag\in(u_i,u_{i+1})$ for some $1\leq i<d$ and $q^\dag\in \partial g(v^\dag)$. Then we have
    \begin{equation}\label{eq91}
        d_\calG^{q^\dag}(v,v^\dag)=0\qquad\text{for any}\quad v\in[u_i,u_{i+1}].
    \end{equation} 
\end{lemma}
\begin{proof}
    By the definition of the Bregman distance and the characterization \eqref{eq:multibang:dg} of $\partial g(v^\dag)$ (which is single-valued under the assumption on $v^\dag$), we directly obtain
    \begin{equation}
        \begin{multlined}
            d_g^{q^\dag}(v,v^\dag)=\frac{1}{2}\left[(u_i+u_{i+1})v-u_iu_{i+1}\right]-\frac{1}{2}[(u_i+u_{i+1})v^\dag-u_iu_{i+1}]\\
            -\frac{1}{2}(u_i+u_{i+1})(v-v^\dag)=0
        \end{multlined}
    \end{equation}
    for any $v\in [u_i,u_{i+1}]$.
\end{proof}

\Cref{lem:conv_pointwise} allows us to translate the weak convergence from \cref{thm:convergence} to pointwise convergence, which is the main result of our work.
\begin{theorem}\label{thm:convergence_pointwise}
    Assume the conditions of \cref{thm:convergence} hold. If $u^\dag(x)\in\{u_1,\dots,u_d\}$ almost everywhere, the subsequence $u_{\alpha_n}^{\delta_n}\to u^\dag$ pointwise almost everywhere.
\end{theorem}
\begin{proof} 
    From \cref{thm:convergence}, we obtain a subsequence $\{u_n\}_{n\in\N}$ of $\{u_{\alpha_n}^{\delta_n}\}_{n\in\N}$ converging weakly to $u^\dag$. Since $\calG$ is convex and lower semicontinuous, we 
    have that
    \begin{equation}\label{eq:convergence_pointwise1}
        \calG(u^\dag)\leq \liminf_{n\to\infty}\calG(u_n)\leq \lim_{n\to\infty}\calG(u_n).
    \end{equation}
    By the minimizing properties of $\{u_n\}_{n\in\N}$ and the nonnegativity of the discrepancy term, we further obtain that
    \begin{equation}
        \alpha_n\calG(u_n) \leq \frac{1}{2}\norm{Ku_n-y^{\delta_n}}_{Y}^2+\alpha_n\calG(u_n)\leq \frac{\delta_n^2}{2}+\alpha_n\calG(u^\dag).
    \end{equation}
    Dividing this inequality by $\alpha_n$ and passing to the limit $n\to \infty$, the assumption on $\alpha_n$ from \cref{thm:convergence} yields that
    \begin{equation}
        \lim_{n\to\infty}\calG(u_n)\leq \calG(u^\dag),
    \end{equation}
    which combined with \eqref{eq:convergence_pointwise1} gives $\lim_{n\to\infty}\calG(u_n)=\calG(u^\dag)$.  
    Hence, $u_n\wkto u^\dag$ implies that
    $d_{\calG}^{p^\dag}(u_n,u^\dag)\to 0$ for any $p^\dag\in\partial\calG(u^\dag)$. By the pointwise characterization \eqref{eq:bregman_pointwise} and the nonnegativity of Bregman distances, this implies that $d_g^{p^\dag(x)}(u_n(x),u^\dag(x))\to 0$ for almost every $x\in\Omega$.
    Choosing now $p^\dag\in\partial \calG(u^\dag)$ such that \eqref{eq:ass_p} holds for $q^\dag = p^\dag(x)$ and $v^\dag=u^\dag(x)$ almost everywhere, the claim follows from \cref{lem:conv_pointwise}.
\end{proof}
Since $u_n(x)\in[u_1,u_d]$ by construction, the subsequence $\{u_n\}_{n\in\N}$ is bounded in $L^\infty(\Omega)$ and hence also converges strongly in $L^p(\Omega)$ for any $1\leq p<\infty$ by Lebesgue's dominated convergence theorem. We remark that since \cref{lem:conv_pointwise} applied to $u_n(x)$ and $u^\dag(x)$ does not hold uniformly in $\Omega$, we cannot expect that the convergence rates from \cref{thm:conv_apriori,thm:conv_aposteriori} hold pointwise or strongly as well. 

\section{Structure of minimizers}\label{sec:structure}

We now briefly discuss the structure of reconstructions obtained by minimizing the Tikhonov functional in \eqref{eq:tikhonov} for given $y^\delta\in Y$ and fixed $\alpha>0$, based on the necessary optimality conditions for \eqref{eq:tikhonov}. Since the discrepancy term is convex and differentiable, we can apply the sum rule for convex subdifferentials. Furthermore, the standard calculus for Fenchel conjugates and subdifferentials (see, e.g., \cite{Schirotzek:2007}) yields for $\calG_\alpha:=\alpha\calG$ that $\calG_\alpha^*(p)=\alpha\calG^*(\alpha^{-1}p)$ and hence that $p\in \partial\calG_\alpha(u)$ if and only if $u\in \partial\calG_\alpha^*(p)=\partial \calG^*(\tfrac1\alpha p)$. We thus obtain as in \cite{CK:2013} that $\bar u:=u_\alpha^\delta\in L^2(\Omega)$ is a solution to \eqref{eq:tikhonov} if and only if there exists a $\bar p\in L^2(\Omega)$ satisfying
\begin{equation}\label{eq:optimality}
    \left\{\begin{aligned}
            \bar p & = K^{*}(y^\delta - K\bar u)\\
            \bar u &\in \partial\calG_\alpha^*(\bar p) := \begin{cases}
                \{u_i\} &  \bar p(x) \in Q_i,\qquad 1\leq i \leq d,\\
                [u_i,u_{i+1}] &\bar p(x) \in  Q_{i,i+1}\quad\ 1\leq i< d.  
            \end{cases}
    \end{aligned}\right.
\end{equation}
for
\begin{align}
    Q_1 &=\set{q}{q <\tfrac\alpha2 (u_1+u_2)},\\
    Q_i &=\set{q}{\tfrac\alpha2(u_{i-1}+u_i) < q < \tfrac\alpha2(u_{i}+u_{i+1})},\quad  1<i<d,\\
    Q_d &=\set{q}{q >\tfrac\alpha2 (u_{d-1}+u_d)},\\
    Q_{i,i+i}&= \set{q}{q = \tfrac\alpha2(u_{i}+u_{i+1})}, \qquad\qquad\qquad\quad 1\leq i < d.
\end{align}
Here we have made use of the pointwise characterization in \eqref{eq:multibang:dgstar} and reformulated the case distinction in terms of $\bar p(x)$ instead of $\frac1\alpha \bar p(x)$.

First, we obtain directly from \eqref{eq:optimality} the desired structure of the reconstruction $\bar u$: Apart from a singular set 
\begin{equation}
    \mathcal{S}:=\set{x\in \Omega}{\bar p(x) = \tfrac\alpha2(u_i+u_{i+1})\text{ for some }1\leq i <d},
\end{equation}
we always have $\bar u(x) \in \{u_1,\dots,u_d\}$. For operators $K$ where $K^*w$ cannot be constant on a set of positive measure unless $w=0$ locally (as is the case for many operators involving solutions to partial differential equations; see \cite[Prop.~2.3]{CK:2013}) and $y^\delta \notin\ran K$, the singular set $\mathcal{S}$ has zero measure  and hence the ``multi-bang'' structure $\bar u\in\{u_1,\dots,u_d\}$ almost everywhere can be guaranteed \emph{a priori} for any $\alpha>0$.

Furthermore, we point out that the regularization parameter $\alpha$ only enters via the case distinction. In particular, increasing $\alpha$ shifts the conditions on $\bar u(x)$ such that the smaller values among the $u_i$ become more preferred. In fact, if $\bar p$ is bounded, we can expect that there exists an $\alpha_0>0$ such that $\bar u\equiv u_1$ for all $\alpha>\alpha_0$. Conversely, for $\alpha\to 0$, the second line of \eqref{eq:optimality} reduces to
\begin{equation}
    \bar u(x) \in
    \begin{cases} 
        \{u_1\} &\text{if } \bar p(x)<0,\\
        \{u_d\} &\text{if } \bar p(x)>0,\\
        [u_1,u_d] &\text{if }\bar p(x)=0,
    \end{cases}
\end{equation}
i.e., \eqref{eq:optimality} coincides with the well-known optimality conditions for bang-bang control problems; see, e.g., \cite[Lem.~2.26]{Troeltzsch}.
Since in the context of inverse problems, we only have $\alpha=\alpha(\delta)\to 0$ if $\delta \to 0$, the limit system \eqref{eq:optimality} will contain consistent data and hence $\bar p\equiv 0$. This allows recovery of $u^\dag(x)\in \{u_2,\dots,u_{d-1}\}$ on a set of positive measure, consistent with \cref{thm:convergence}. However, if $u^\dag(x) \in \{u_1,\dots,u_d\}$ does not hold almost everywhere, we can only expect weak and not strong convergence, cf.~\cite[Prop.~5.10\,(ii)]{CTW:2016}.

\section{Numerical solution}\label{sec:semismooth}

In this section we address the numerical solution of the Tikhonov minimization problem \eqref{eq:tikhonov} for given $y^\delta \in Y$ and $\alpha>0$, following \cite{CK:2015}. For the sake of presentation, we omit the dependence on $\alpha$ and $\delta$ from here on.
We start from the necessary (and, due to convexity, sufficient) optimality conditions \eqref{eq:optimality}. To apply a semismooth Newton method, we replace  the subdifferential inclusion $\bar u\in \partial \calG_\alpha^*(\bar p)$ by its single-valued Moreau--Yosida regularization, i.e., we consider for $\gamma>0$ the regularized optimality conditions 
\begin{equation}\label{eq:opt_reg}
    \left\{\begin{aligned}
            p_\gamma & = K^{*}(y^\delta - Ku_\gamma)\\
            u_\gamma & = (\partial\calG_\alpha^*)_\gamma(p_\gamma).
    \end{aligned}\right.
\end{equation}
The Moreau--Yosida regularization can also be expressed as 
\begin{equation}
    H_\gamma:=(\partial\calG_\alpha^*)_\gamma = \partial(\calG_{\alpha,\gamma})^*
\end{equation}
for 
\begin{equation}
    \calG_{\alpha,\gamma}(u) := \alpha\calG(u) + \frac\gamma2\norm{u}_{L^2(\Omega)}^2,
\end{equation}
see, e.g., \cite[Props.~13.21, 12.29]{Bauschke:2011}. This implies that  for $(u_\gamma,p_\gamma)$ satisfying \eqref{eq:opt_reg}, $u_\gamma$ is a solution to the strictly convex problem
\begin{equation}
    \min_{u\in L^2(\Omega)} \frac12\norm{Ku-y^\delta}_Y^2 + \alpha\calG(u) + \frac\gamma2\norm{u}_{L^2(\Omega)}^2,
\end{equation}
so that existence of a solution can be shown by the same arguments as for \eqref{eq:tikhonov}. Note that by regularizing the conjugate subdifferential, we have not smoothed the nondifferentiability but merely made the functional (more) strongly convex. The regularization of $\calG_\alpha^*$ instead of $\calG^*$ also ensures that the regularization is robust for $\alpha\to0$. From \cite[Prop.~4.1]{CK:2015}, we obtain the following convergence result.
\begin{proposition}
    The family $\{u_{\gamma}\}_{\gamma>0}$ satisfying \eqref{eq:opt_reg} contains at least one subsequence $\{u_{\gamma_n}\}_{n\in\N}$ converging to a global minimizer of \eqref{eq:tikhonov} as $n\to\infty$. Furthermore, for any such subsequence, the convergence is strong.
\end{proposition}

From \cite[Appendix~\textsc{a}.2]{CIK:2014} we further obtain the pointwise characterization 
\begin{equation}\label{eq:hgamma}
    [H_\gamma(p)](x) =
    \begin{cases}
        u_i & \text{if }p(x)\in Q^\gamma_i,\qquad 1\leq i\leq d,\\
        \tfrac1\gamma(p(x)-\tfrac\alpha2(u_i+u_{i+1})) & \text{if }p(x) \in Q^\gamma_{i,i+1},\quad 1\leq i < d,
    \end{cases}
\end{equation}
where
\begin{align}
    Q_1^\gamma &= \set{q}{q< \tfrac\alpha2\left((1+ {2\gamma})u_1+u_2\right)},\\
    Q_i^\gamma &= \set{q}{\tfrac\alpha2\left(u_{i-1} + (1 + {2\gamma})u_i\right) < q < \tfrac\alpha2\left((1+ {2\gamma})u_i+u_{i+1}\right)}
    \quad \text{ for } 1<i<d,\\
    Q_d^\gamma &= \set{q}{\tfrac\alpha2\left(u_{d-1} + (1 + {2\gamma})u_d\right) < q},\\
    Q_{i,i+1}^\gamma &= \set{q}{\tfrac\alpha2\left((1 + {2\gamma})u_i+u_{i+1}\right) \leq q \leq \tfrac\alpha2\left(u_i+(1+ {2\gamma})u_{i+1}\right)} \quad\text{for } 1 \leq i < d.
\end{align}
Since $H_\gamma$ is a superposition operator defined by a Lipschitz continuous and piecewise differentiable scalar function, $H_\gamma$ is Newton-differentiable from $L^r(\Omega)\to L^2(\Omega)$ for any $r>2$; see, e.g., \cite[Example 8.12]{Kunisch:2008a} or \cite[Theorem 3.49]{Ulbrich:2011}. A Newton derivative at $p$ in direction $h$ is given pointwise almost everywhere by
\begin{equation}
    [D_N H_\gamma(p)h](x) =
    \begin{cases}
        \frac1\gamma h(x) & \text{if }p(x)\in Q^\gamma_{i,i+1},\quad 1\leq i < d,\\
        0 &\text{else.}
    \end{cases}
\end{equation}

Hence if the range of $K^*$ embeds into $L^r(\Omega)$ for some $r>2$ (which is the case, e.g., for many convolution operators and solution operators for partial differential equations) and the semismooth Newton step is uniformly invertible, the corresponding Newton iteration converges locally superlinearly. We address this for the concrete example considered in the next section. In practice, the local convergence can be addressed by embedding the Newton method into a continuation strategy, i.e., starting for $\gamma$ large and then iteratively reducing $\gamma$, using the previous solution as a starting point.

\section{Numerical examples}\label{sec:examples}

We illustrate the proposed approach for an inverse source problem for the Poisson equation, i.e., we choose $K=A^{-1}:L^2(\Omega)\to L^2(\Omega)$ for $\Omega=[0,1]^2$ and $A=-\Delta$ together with homogeneous boundary conditions. We note that since $\Omega$ is a Lipschitz domain, we have that $\ran A^{-*}=\ran A^{-1} = H^2(\Omega)\cap H^1_0(\Omega)$, and hence this operator satisfies the conditions discussed in \cref{sec:structure} that guarantee that $u_\alpha^\delta(x)\in\{u_1,\dots,u_d\}$ almost everywhere if $y^\delta\notin\ran K$; see \cite[Prop.~2.3]{CK:2013}.
For the computational results below, we use a finite element discretization on a uniform triangular grid with $256\times 256$ vertices.

The specific form of $K$ can be used to reformulate the optimality condition (and hence the Newton system) into a more convenient form. Introducing $y_\gamma = A^{-1}u_\gamma$ and eliminating $u_\gamma$ using the second relation of \eqref{eq:opt_reg}, we obtain as in \cite{CK:2013} the equivalent system
\begin{equation}\label{eq:opt_reg_ref}
    \left\{\begin{aligned}
            A^*p_\gamma + y_\gamma - y^\delta &=0,\\
            Ay_\gamma - H_\gamma(p_\gamma) &=0.
    \end{aligned}\right.
\end{equation}
Setting $V:=H^1_0(\Omega)$, we can consider this as an equation from $V\times V$ to $V^*\times V^*$, which due to the embedding $V\hookrightarrow L^p(\Omega)$ for $p>2$ provides the necessary norm gap for Newton differentiability of $H_\gamma$.
By the chain rule for Newton derivatives from, e.g., \cite[Lem.~8.4]{Kunisch:2008a}, the corresponding Newton step therefore consists of solving for $(\delta y,\delta p)\in V\times V$ given $(y^k,p^k)\in V\times V$ in
\begin{equation}\label{eq:newton_step}
    \begin{pmatrix} 
        \Id & A^*\\
        A & -D_N H_\gamma(p^k) 
    \end{pmatrix}
\begin{pmatrix}\delta y \\ \delta p\end{pmatrix}
    = 
    -\begin{pmatrix} A^*p^k +  y - y^\delta\\ Ay^k - H_\gamma( p^k)
    \end{pmatrix}
\end{equation}
and setting
\begin{equation}
    y^{k+1} = y^k + \delta y,\qquad p^{k+1} = p^k + \delta p.
\end{equation}
Note that the reformulated Newton matrix is symmetric, which in general is not the case for nonsmooth equations. Following \cite[Prop.~4.3]{CK:2013}, the Newton step \eqref{eq:newton_step} is uniformly boundedly invertible, from which local superlinear convergence to a solution of \eqref{eq:opt_reg_ref} follows. 

In practice, we include the continuation strategy described above as well as a simple backtracking line search based on the residual norm in \eqref{eq:opt_reg_ref} to improve robustness. Since the forward operator is linear and $H_\gamma$ is piecewise linear, the semi-smooth Newton method has the following finite termination property: If $H_\gamma(p^{k+1})=H_\gamma(p^k)$, then $(y^{k+1},p^{k+1})$ satisfy \eqref{eq:opt_reg_ref}; cf.~\cite[Rem.~7.1.1]{Kunisch:2008a}. We then recover $u^{k+1}=H_\gamma(p^{k+1})$. In the implementation, we also terminate if more than $100$ Newton iterations are performed, in which case the continuation is also terminated and the last successful iterate is returned. Otherwise we terminate if $\gamma<10^{-12}$. In all results reported below, the continuation is terminated successfully.
The implementation of this approach used to obtain the following results can be downloaded from \url{https://github.com/clason/discreteregularization}.

\bigskip

The first example illustrates the convergence behavior of the Tikhonov regularization. Here, the true parameter is chosen as
\begin{equation}\label{eq:exact_data_mb}
    \begin{aligned}[t]
        u^\dag(x) = u_1 &+ u_2 \,\chi_{\{x:(x_1 - 0.45)^2 + (x_2 - 0.55)^2 < 0.1\}}(x)\\
                        &+ (u_3-u_2)\,  \chi_{\{x:(x_1 - 0.4)^2 + (x_2 - 0.6)^2 < 0.02\}}(x)
    \end{aligned}
\end{equation}
for $(u_1,u_2,u_3)=(0,0.1,0.15)$; see \cref{fig:recon_15:exact}. (This might correspond to, e.g., material properties of background, healthy tissue, and tumor, respectively.)
The noisy data is constructed pointwise via
\begin{equation}
    y^\delta = y^\dag + (\tilde\delta \norm{y^\dag}_{\infty})\xi,
\end{equation}
where $\xi$ is a vector of identically and independently normally distributed random variables with mean $0$ and variance $1$, and $\tilde \delta \in\{2^0,\dots,2^{-20}\}$. For each value of $\tilde \delta$, the corresponding regularization parameter $\alpha$ is chosen according to the discrepancy principle \eqref{eq:morozov} with $\tau = 1.1$.
Details on the convergence history are reported in \cref{tab:recon_15}, which shows the effective noise level $\delta:=\norm{y^\delta-y^\dag}_2$, the parameter $\alpha$ selected as satisfying the Morozov discrepancy principle, the $L^2$-error $e_2:=\norm{u_{\alpha}^\delta-u^\dag}_2$ and the $L^\infty$-error $e_\infty:=\norm{u_\alpha^\delta-u^\dag}_\infty$. First, we note that the \emph{a posteriori} choice approximately follows the \emph{a priori} choice $\alpha \sim \delta$. Similarly, for larger values of $\delta$, the $L^2$-error behaves as $e_2\sim \delta$, which is no longer true for $\delta\to 0$ (and cannot be expected due to the nonsmooth regularization). The $L^\infty$-error $e_\infty$ is initially dominated by the jump in admissible parameter values: As long as there is a single point $x\in\Omega$ with $u_\alpha^\delta(x) = u_i \neq u_j = u^\dag(x)$, we necessarily have $e_\infty\geq\min_{1\leq i<d} u_{i+1}-u_{i}$. (Recall that we do not have a convergence rate and thus an error bound for pointwise convergence.) Later, $e_\infty$ becomes smaller than this threshold value, which indicates that apart from points in the regularized singular set (i.e., where $p_\gamma(x)\in Q^\gamma_{i,i+1}$, which in these cases happens for $~20$ out of $256\times256$ vertices), the reconstruction is exact. Here we point out that since $\gamma$ is independent of $\alpha$, the Moreau--Yosida regularization for fixed $\gamma$ becomes more and more active as $\alpha\to 0$. Nevertheless, in all cases $\gamma \ll \alpha$, and hence the multi-bang regularization dominates.

The pointwise convergence can also be seen clearly from \cref{fig:recon_15}, which shows the true parameter $u^\dag$ together with three representative reconstructions for different noise levels. It can be seen that for large noise, the corresponding large regularization suppresses the smaller inclusion; see \cref{fig:recon_15:4}. This is consistent with the discussion at the end of \cref{sec:structure}. For smaller noise, the inclusion is recovered well (\cref{fig:recon_15:7}), and for $\delta\approx 3.69\cdot 10^{-4}$, the reconstruction is visually indistinguishable from the true parameter (\cref{fig:recon_15:13}).

\begin{figure}[p]
    \vspace*{-0.25cm}
    \centering
    \begin{subfigure}[t]{0.495\textwidth}
        \includegraphics[width=\textwidth]{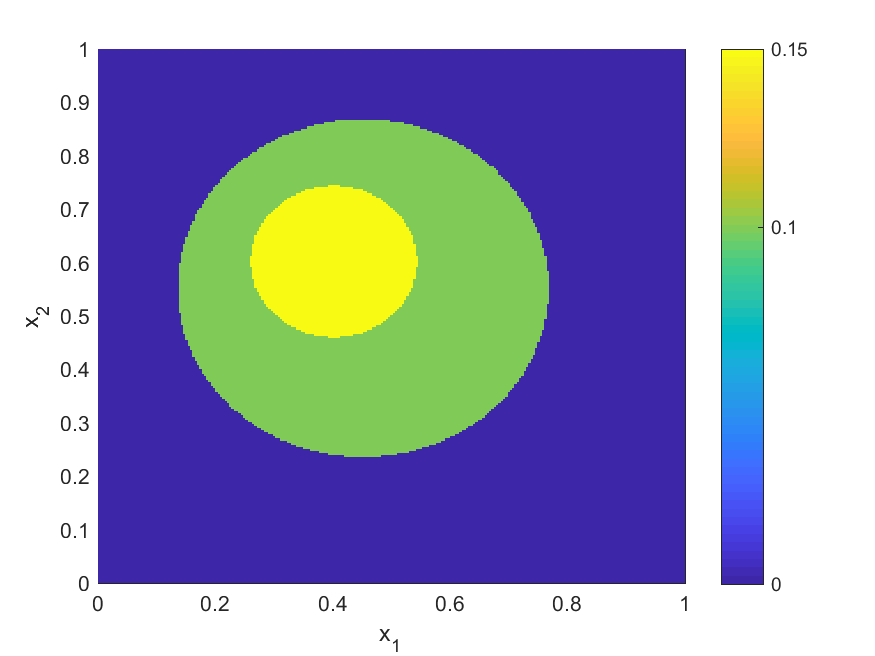}
        \caption{$u^\dag$}
        \label{fig:recon_15:exact}
    \end{subfigure}
    \hfill
    \begin{subfigure}[t]{0.495\textwidth}
        \includegraphics[width=\textwidth]{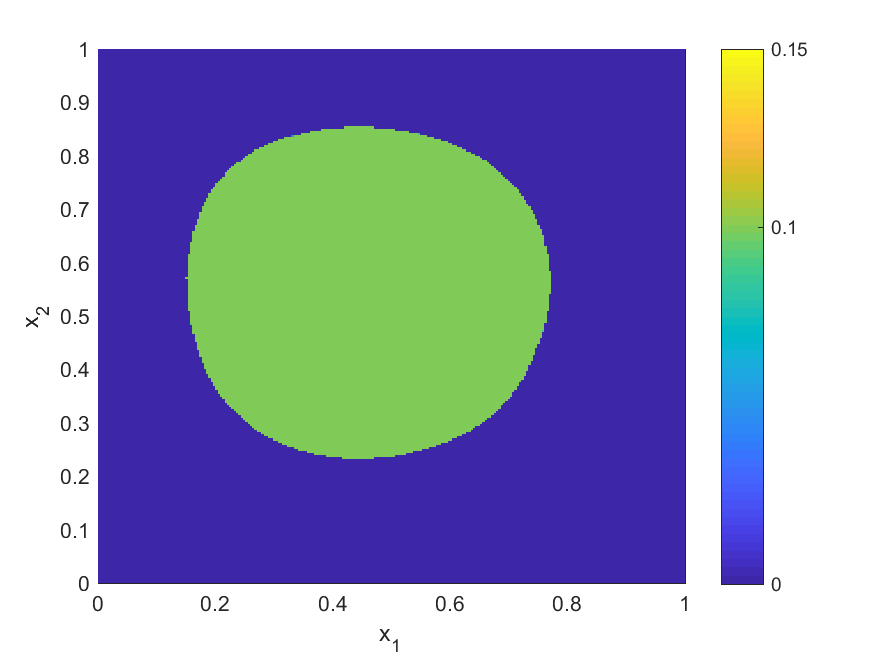}
        \caption{$u_\alpha^\delta$ for $\delta \approx1.89\cdot10^{-1}$}
        \label{fig:recon_15:4}
    \end{subfigure}

    \begin{subfigure}[t]{0.495\textwidth}
        \includegraphics[width=\textwidth]{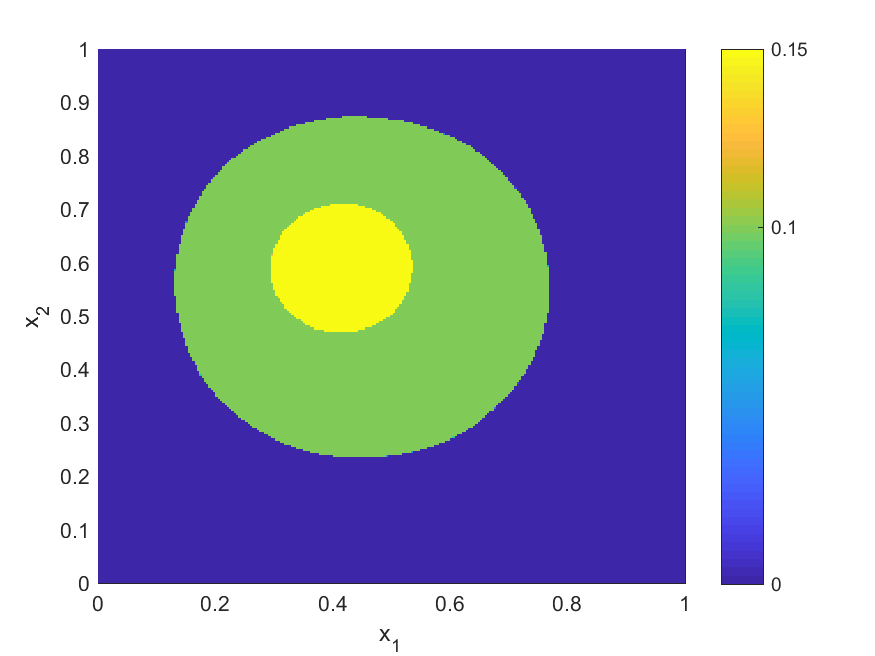}
        \caption{$u_\alpha^\delta$ for $\delta \approx 2.37\cdot10^{-2}$}
        \label{fig:recon_15:7}
    \end{subfigure}
    \hfill
    \begin{subfigure}[t]{0.495\textwidth}
        \includegraphics[width=\textwidth]{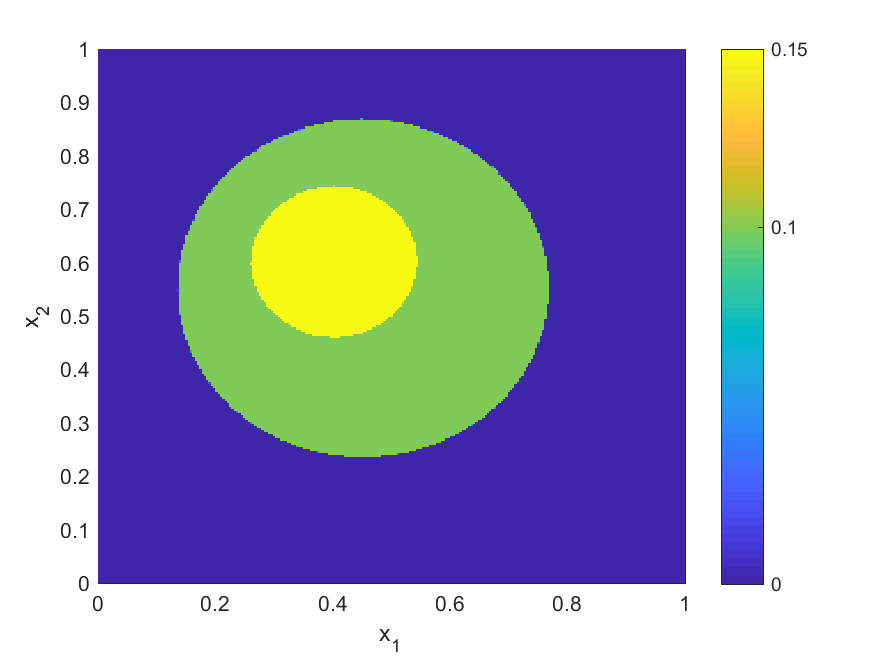}
        \caption{$u_\alpha^\delta$ for $\delta \approx 3.69\cdot10^{-4}$}
        \label{fig:recon_15:13}
    \end{subfigure}
    \caption{True parameter $u^\dag$ for $u_3=0.15$ and reconstructions $u_\alpha^\delta$ for different values of $\delta$}
    \label{fig:recon_15}
\end{figure}
\begin{table}[p]
    \caption{Convergence behavior as $\delta \to 0$ for $u_3=0.15$: noise level $\delta$, regularization parameter $\alpha$, $L^2$-error $e_2$, $L^\infty$-error $e_\infty$}\label{tab:recon_15}
    \begin{minipage}[t]{0.5\textwidth}
        \centering
        \begin{tabular}[t]{S S S S}
            \toprule
            {$\delta$}  & {$\alpha$} & {${e}_2$} & {$e_\infty$}\\
            \midrule
            1.52e+0     & 1.00e-2    & 1.60e+1   & 1.50e-1\\
            7.59e-1     & 1.25e-3    & 8.64e+0   & 1.00e-1\\
            3.78e-1     & 6.25e-4    & 6.18e+0   & 1.00e-1\\
            1.89e-1     & 3.13e-4    & 4.26e+0   & 1.00e-1\\
            9.48e-2     & 7.81e-5    & 4.32e+0   & 1.00e-1\\
            4.73e-2     & 3.91e-5    & 3.67e+0   & 1.00e-1\\
            2.37e-2     & 1.95e-5    & 2.97e+0   & 1.00e-1\\
            1.19e-2     & 9.77e-6    & 2.33e+0   & 1.00e-1\\
            5.90e-3     & 4.88e-6    & 1.76e+0   & 1.00e-1\\
            2.95e-3     & 2.44e-6    & 1.33e+0   & 1.00e-1\\
            1.49e-3     & 1.22e-6    & 9.47e-1   & 1.00e-1\\
            \bottomrule
        \end{tabular}
    \end{minipage}
    \hfill
    \begin{minipage}[t]{0.5\textwidth}
        \centering
        \begin{tabular}[t]{S S S S}
            \toprule
            {$\delta$}  & {$\alpha$} & {${e}_2$} & {$e_\infty$}\\
            \midrule
            7.44e-4     & 6.10e-7    & 6.86e-1   & 1.00e-1\\
            3.69e-4     & 3.05e-7    & 4.74e-1   & 1.00e-1\\
            1.85e-4     & 1.53e-7    & 2.91e-1   & 7.82e-2\\
            9.28e-5     & 7.63e-8    & 2.27e-1   & 7.67e-2\\
            4.64e-5     & 3.81e-8    & 1.29e-1   & 5.73e-2\\
            2.32e-5     & 1.91e-8    & 9.19e-2   & 4.91e-2\\
            1.16e-5     & 9.54e-9    & 9.32e-2   & 4.03e-2\\
            5.79e-6     & 4.77e-9    & 4.61e-2   & 2.30e-2\\
            2.89e-6     & 2.38e-9    & 1.13e-1   & 5.00e-2\\
            1.44e-6     & 5.96e-10   & 1.70e-2   & 4.39e-3\\\\
            \bottomrule
        \end{tabular}
    \end{minipage}
\end{table}

The behavior is essentially the same if we set $(u_1,u_2,u_3)=(0,0.1,0.11)$ in \eqref{eq:exact_data_mb} (i.e., a contrast of $10\%$ instead of $50\%$ for the inner inclusion), demonstrating the robustness of the multi-bang regularization; see \cref{fig:recon_11} and \cref{tab:recon_11}.

\begin{figure}[p]
    \vspace*{-0.25cm}
    \centering
    \begin{subfigure}[t]{0.495\textwidth}
        \includegraphics[width=\textwidth]{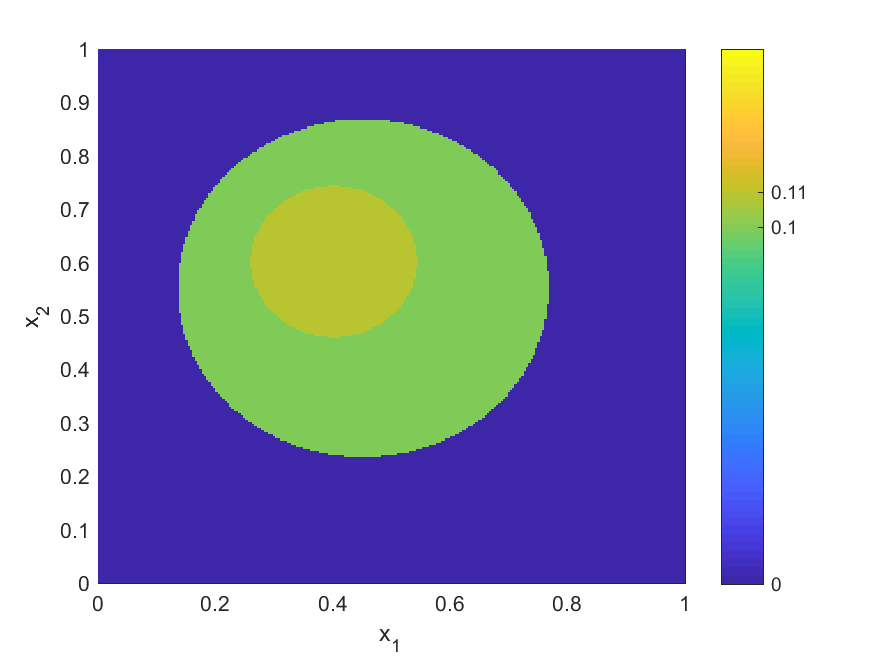}
        \caption{$u^\dag$}
        \label{fig:recon_11:exact}
    \end{subfigure}
    \hfill
    \begin{subfigure}[t]{0.495\textwidth}
        \includegraphics[width=\textwidth]{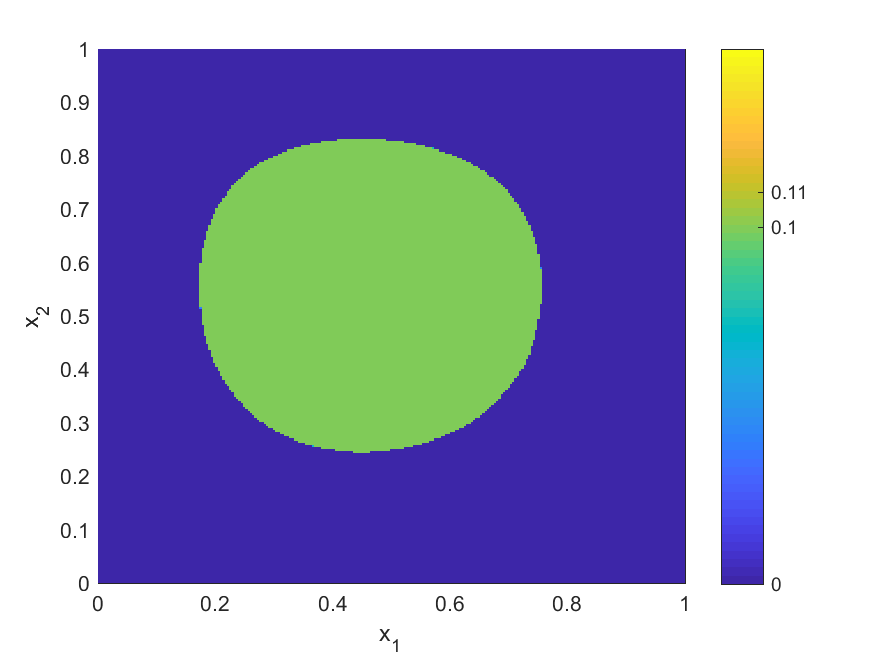}
        \caption{$u_\alpha^\delta$ for $\delta \approx1.68\cdot10^{-1}$}
        \label{fig:recon_11:4}
    \end{subfigure}

    \begin{subfigure}[t]{0.495\textwidth}
        \includegraphics[width=\textwidth]{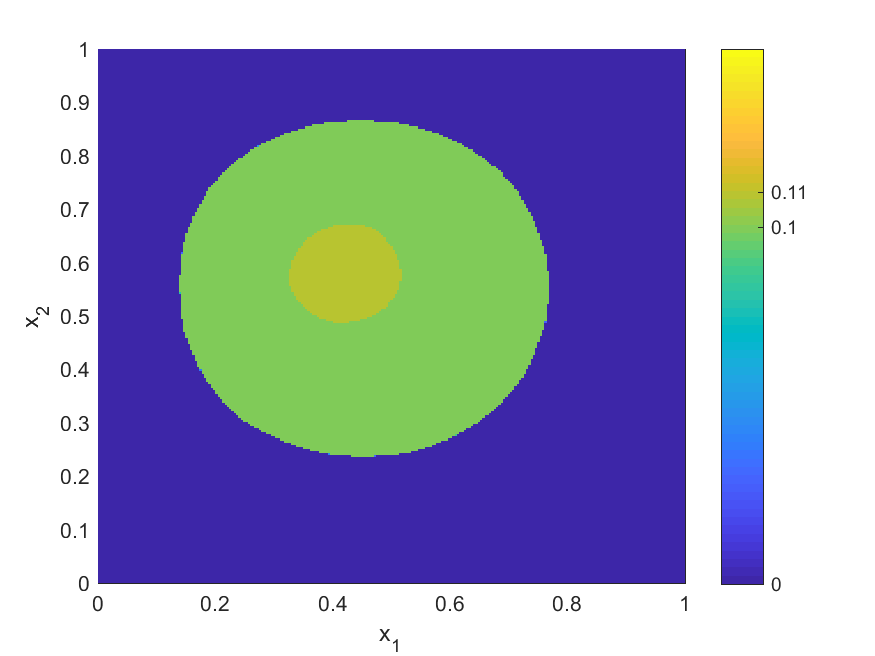}
        \caption{$u_\alpha^\delta$ for $\delta \approx 2.17\cdot10^{-2}$}
        \label{fig:recon_11:7}
    \end{subfigure}
    \hfill
    \begin{subfigure}[t]{0.495\textwidth}
        \includegraphics[width=\textwidth]{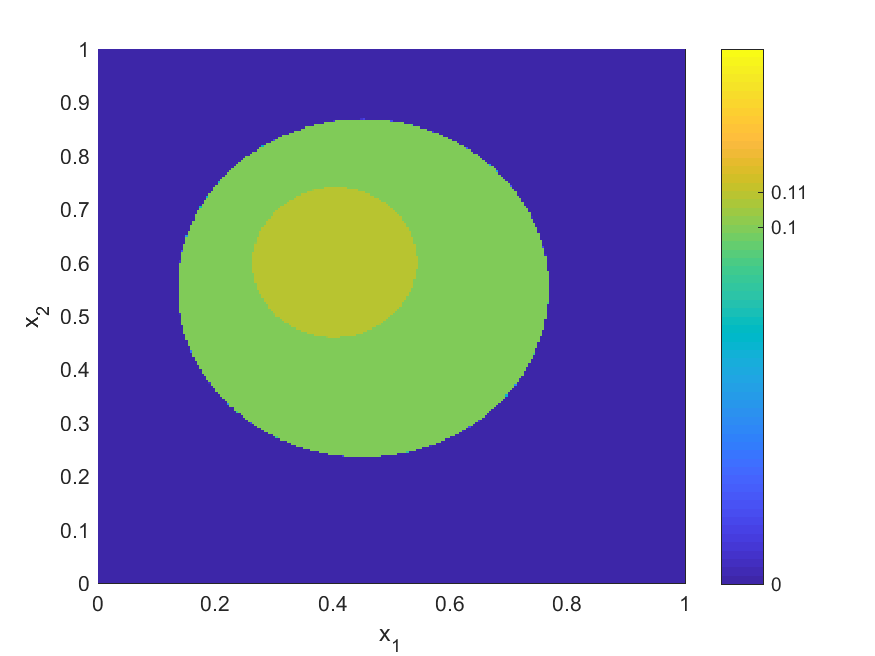}
        \caption{$u_\alpha^\delta$ for $\delta \approx 3.29\cdot10^{-4}$}
        \label{fig:recon_11:13}
    \end{subfigure}
    \caption{True parameter $u^\dag$ for $u_3=0.11$ and reconstructions $u_\alpha^\delta$ for different values of $\delta$}
    \label{fig:recon_11}
\end{figure}
\begin{table}[p]
    \caption{Convergence behavior as $\delta \to 0$ for $u_3=0.11$: noise level $\delta$, regularization parameter $\alpha$, $L^2$-error $e_2$, $L^\infty$-error $e_\infty$}\label{tab:recon_11}
    \begin{minipage}[t]{0.5\textwidth}
        \centering
        \begin{tabular}[t]{S S S S}
            \toprule
            {$\delta$}  & {$\alpha$} & {${e}_2$} & {$e_\infty$}\\
            \midrule
            1.34e+0     & 2.50e-3    & 1.16e+0   & 1.10e-1 \\
            6.73e-1     & 1.25e-3    & 9.13e+0   & 1.00e-1 \\
            3.36e-1     & 6.25e-4    & 6.89e+0   & 1.00e-1 \\
            1.68e-1     & 3.13e-4    & 4.91e+0   & 1.00e-1 \\
            8.41e-2     & 1.56e-4    & 3.27e+0   & 1.00e-1 \\
            4.20e-2     & 3.91e-5    & 1.90e+0   & 1.00e-1 \\
            2.17e-2     & 1.95e-5    & 1.57e+0   & 1.00e-1 \\
            1.05e-3     & 9.77e-6    & 1.19e+0   & 1.00e-1 \\
            5.25e-3     & 4.88e-6    & 9.81e-1   & 1.00e-1 \\
            2.64e-3     & 2.44e-6    & 8.14e-1   & 1.00e-1 \\
            1.32e-4     & 1.22e-6    & 6.70e-1   & 1.00e-1 \\
            \bottomrule
        \end{tabular}
    \end{minipage}
    \hfill
    \begin{minipage}[t]{0.5\textwidth}
        \centering
        \begin{tabular}[t]{S S S S}
            \toprule
            {$\delta$}  & {$\alpha$} & {${e}_2$} & {$e_\infty$}\\
            \midrule
            6.56e-4     & 6.10e-7    & 4.55e-1   & 1.00e-1 \\
            3.29e-4     & 3.05e-7    & 2.94e-1   & 1.00e-1 \\
            1.64e-4     & 1.53e-7    & 2.20e-1   & 6.15e-2 \\
            8.27e-5     & 7.63e-8    & 1.87e-1   & 8.55e-2 \\
            4.11e-5     & 3.81e-8    & 6.75e-2   & 3.35e-2\\
            2.07e-5     & 1.91e-8    & 4.34e-2   & 1.44e-2 \\
            1.03e-5     & 9.54e-9    & 3.72e-2   & 1.46e-2 \\
            5.12e-6     & 4.77e-9    & 3.29e-2   & 1.31e-2 \\
            2.56e-6     & 2.38e-9    & 3.85e-2   & 1.00e-2 \\
            1.29e-6     & 2.98e-10   & 1.65e-1   & 1.79e-2 \\\\
            \bottomrule
        \end{tabular}
    \end{minipage}
\end{table}

To illustrate the behavior if the true parameter does not satisfy the assumption $u^\dag\in\{u_1,\dots,u_d\}$ almost everywhere, we repeat the above for 
\begin{equation}\label{eq:exact_data_nonmb}
    \begin{aligned}[t]
        u^\dag(x) = u_1 &+ u_2 \,\chi_{\{x:(x_1 - 0.45)^2 + (x_2 - 0.55)^2 < 0.1\}}(x)\\
                        &+ (u_3-u_2)(1-x_1)\,  \chi_{\{x:(x_1 - 0.4)^2 + (x_2 - 0.6)^2 < 0.02\}}(x)
    \end{aligned}
\end{equation}
with $(u_1,u_2,u_3)=(0,0.1,0.12)$;
see \cref{fig:recon_12_lin:exact}. While for large noise level and regularization parameter value, the multi-bang regularization behaves as before (see \cref{fig:recon_12_lin:7}), the reconstruction for smaller noise and regularization (\cref{fig:recon_12_lin:13}) shows the typical checkerboard pattern expected from weak but not strong convergence; cf.~\cite[Rem.~4.2]{CK:2013}. Nevertheless, as $\delta\to 0$, we still observe convergence to the true parameter; see \cref{fig:recon_12_lin:21} and \cref{tab:recon_12}.

\begin{figure}[p]
    \vspace*{-0.25cm}
    \centering
    \begin{subfigure}[t]{0.495\textwidth}
        \includegraphics[width=\textwidth]{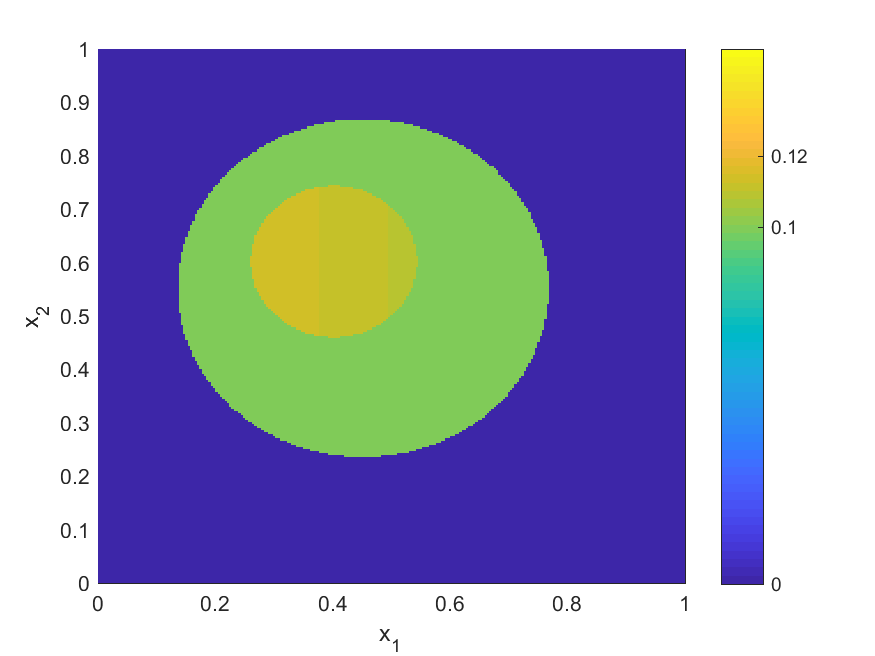}
        \caption{$u^\dag$}
        \label{fig:recon_12_lin:exact}
    \end{subfigure}
    \hfill
    \begin{subfigure}[t]{0.495\textwidth}
        \includegraphics[width=\textwidth]{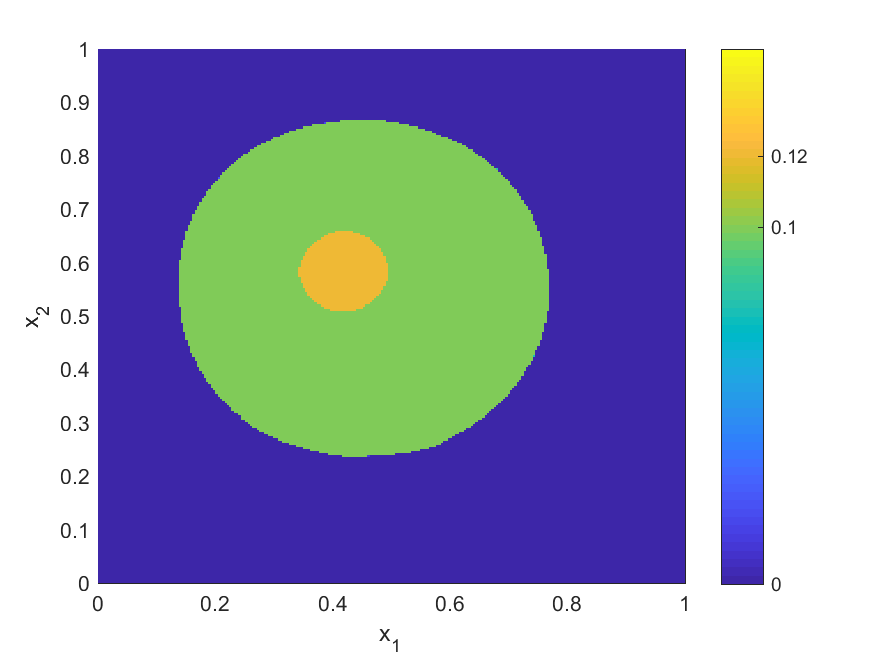}
        \caption{$u_\alpha^\delta$ for $\delta \approx 2.11\cdot10^{-2}$}
        \label{fig:recon_12_lin:7}
    \end{subfigure}

    \begin{subfigure}[t]{0.495\textwidth}
        \includegraphics[width=\textwidth]{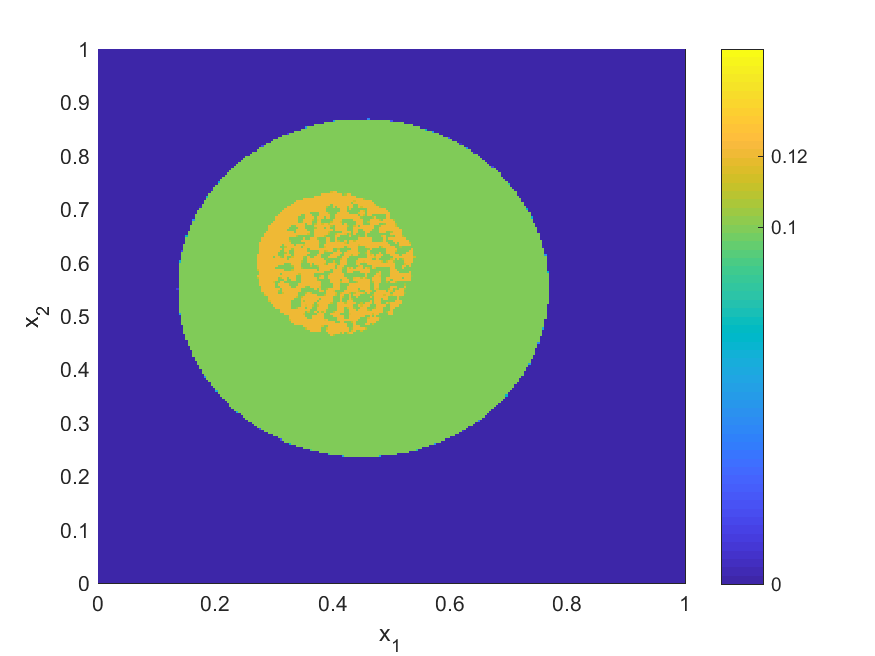}
        \caption{$u_\alpha^\delta$ for $\delta \approx 3.29\cdot10^{-4}$}
        \label{fig:recon_12_lin:13}
    \end{subfigure}
    \hfill
    \begin{subfigure}[t]{0.495\textwidth}
        \includegraphics[width=\textwidth]{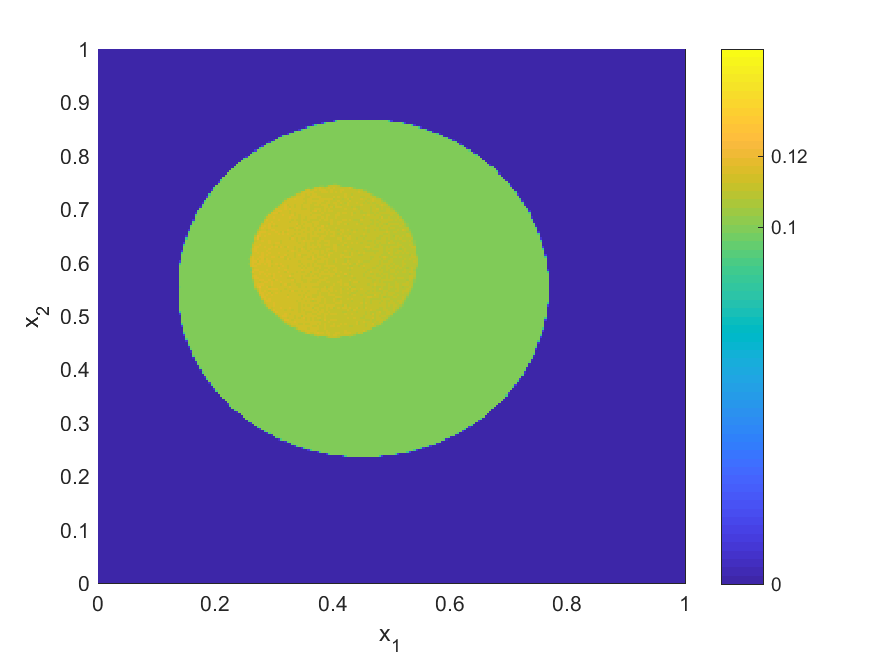}
        \caption{$u_\alpha^\delta$ for $\delta \approx 1.29\cdot10^{-6}$}
        \label{fig:recon_12_lin:21}
    \end{subfigure}
    \caption{True parameter $u^\dag$ and reconstructions $u_\alpha^\delta$ for different values of $\delta$}
    \label{fig:recon_12_lin}
\end{figure}
\begin{table}[p]
    \caption{Convergence behavior as $\delta \to 0$ for $u^\dag$: noise level $\delta$, regularization parameter $\alpha$, $L^2$-error $e_2$, $L^\infty$-error $e_\infty$}\label{tab:recon_12}
    \begin{minipage}[t]{0.5\textwidth}
        \centering
        \begin{tabular}[t]{S S S S}
            \toprule
            {$\delta$} & {$\alpha$} & {${e}_2$} & {$e_\infty$}\\
            \midrule
            1.36e+0    & 2.50e-3    & 1.17e+1   & 1.15e-1\\
            6.77e-1    & 1.25e-3    & 9.08e+0   & 1.00e-1\\
            3.39e-1    & 6.25e-4    & 6.84e+0   & 1.00e-1\\
            1.69e-1    & 3.12e-4    & 4.81e+0   & 1.00e-1\\
            8.48e-2    & 1.56e-4    & 3.12e+0   & 1.00e-1\\
            4.22e-2    & 3.91e-5    & 2.03e+0   & 1.00e-1\\
            2.11e-2    & 1.95e-5    & 1.67e+0   & 1.00e-1\\
            1.05e-2    & 9.77e-6    & 1.45e+0   & 1.00e-1\\
            5.29e-3    & 4.88e-6    & 1.29e+0   & 1.00e-1\\
            2.66e-3    & 2.44e-6    & 1.18e+0   & 1.00e-1\\
            1.32e-3    & 1.22e-6    & 9.82e-1   & 1.00e-1\\
            \bottomrule
        \end{tabular}
    \end{minipage}
    \hfill
    \begin{minipage}[t]{0.5\textwidth}
        \centering
        \begin{tabular}[t]{S S S S}
            \toprule
            {$\delta$}  & {$\alpha$} & {${e}_2$} & {$e_\infty$}\\
            \midrule
            6.60e-4     & 6.10e-7    & 8.46e-1   & 1.00e-1\\
            3.29e-4     & 1.53e-7    & 7.23e-1   & 1.00e-1\\
            1.66e-4     & 7.63e-8    & 6.20e-1   & 5.63e-2\\
            8.25e-5     & 3.81e-8    & 6.04e-1   & 5.60e-2\\
            4.12e-5     & 1.91e-8    & 5.69e-1   & 1.83e-2\\
            2.06e-5     & 9.54e-9    & 5.82e-1   & 5.60e-2\\
            1.03e-5     & 4.77e-9    & 4.95e-1   & 5.66e-2\\
            5.15e-6     & 2.38e-9    & 3.39e-1   & 1.47e-2\\
            2.58e-6     & 5.96e-10   & 2.70e-1   & 2.61e-2\\
            1.29e-6     & 3.73e-11   & 1.65e-1   & 1.48e-2\\\\
            \bottomrule
        \end{tabular}
    \end{minipage}
\end{table}

Finally, we address the qualitative dependence of the reconstruction on the regularization parameter $\alpha$. \Cref{fig:recon_d05} shows reconstructions for the true parameter $u^\dag$ from \eqref{eq:exact_data_mb} again with $(u_1,u_2,u_3)=(0,0.1,0.15)$ for an effective noise level $\delta \approx 0.759$ and different values of $\alpha$. First, \cref{fig:recon_d05:3} presents the reconstruction for the value $\alpha=1.25\cdot10^{-3}$, where as before the volume corresponding to $u_2$ is reduced and the inner inclusion corresponding to $u_3$ is suppressed completely. If the parameter is chosen smaller as $\alpha = 10^{-4}$, however, the reconstruction of the outer volume is essentially correct, while the inner inclusion -- although reduced -- is also localized well; see \cref{fig:recon_d05:4}. Visually, this value yields a better reconstruction than the one obtained by the discrepancy principle. The trade-off is a loss of spatial regularity, manifested in more irregular level lines, which becomes even more pronounced for smaller $\alpha = 10^{-5}$; see \cref{fig:recon_d05:5}. This behavior is surprising insofar that the pointwise definition of the multi-bang penalty itself imposes no spatial regularity on the reconstruction at all; as is evident from \eqref{eq:optimality}, any regularity of the solution $\bar u$ is solely due to that of the level sets of $\bar p$ (which in this case has the regularity of a solution to a Poisson equation).

\begin{figure}[t]
    \vspace*{-0.25cm}
    \centering
    \begin{subfigure}[t]{0.495\textwidth}
        \includegraphics[width=\textwidth]{recon_15_exact.png}
        \caption{$u^\dag$}
        \label{fig:recon_d05:exact}
    \end{subfigure}
    \hfill
    \begin{subfigure}[t]{0.495\textwidth}
        \includegraphics[width=\textwidth]{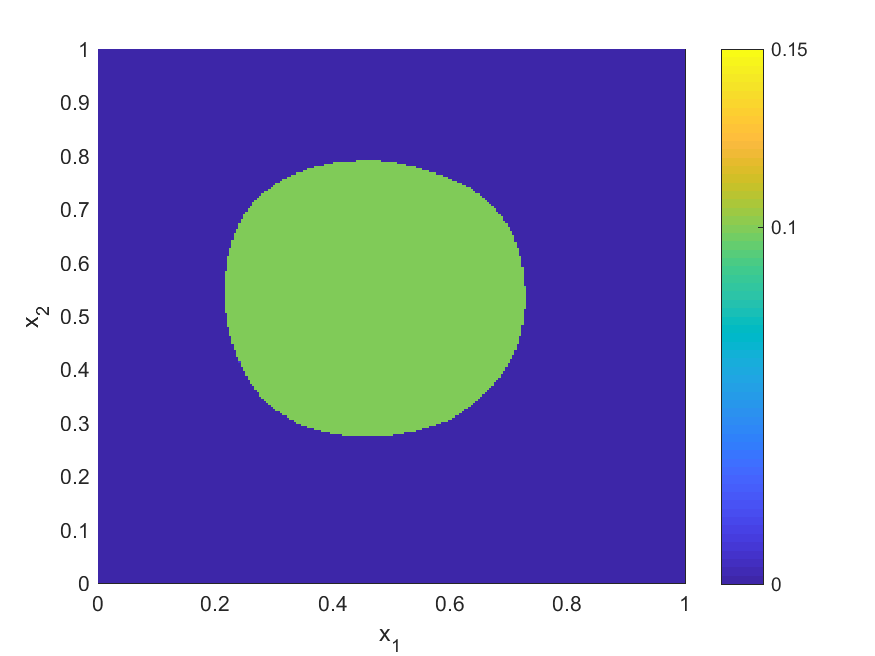}
        \caption{$u_\alpha^\delta$ for $\alpha =1.25\cdot10^{-3} $}
        \label{fig:recon_d05:3}
    \end{subfigure}

    \begin{subfigure}[t]{0.495\textwidth}
        \includegraphics[width=\textwidth]{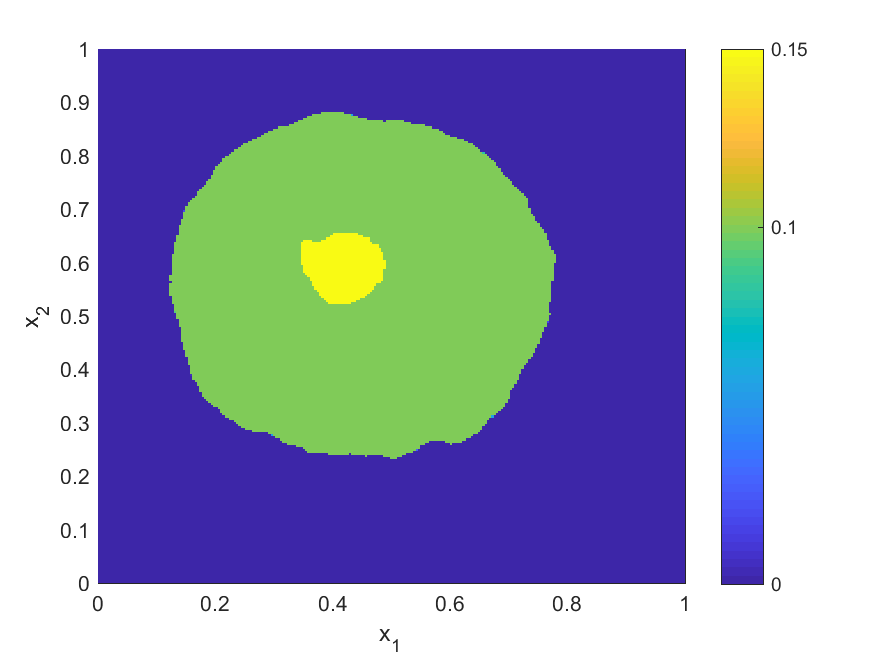}
        \caption{$u_\alpha^\delta$ for $\alpha = 10^{-4}$}
        \label{fig:recon_d05:4}
    \end{subfigure}
    \hfill
    \begin{subfigure}[t]{0.495\textwidth}
        \includegraphics[width=\textwidth]{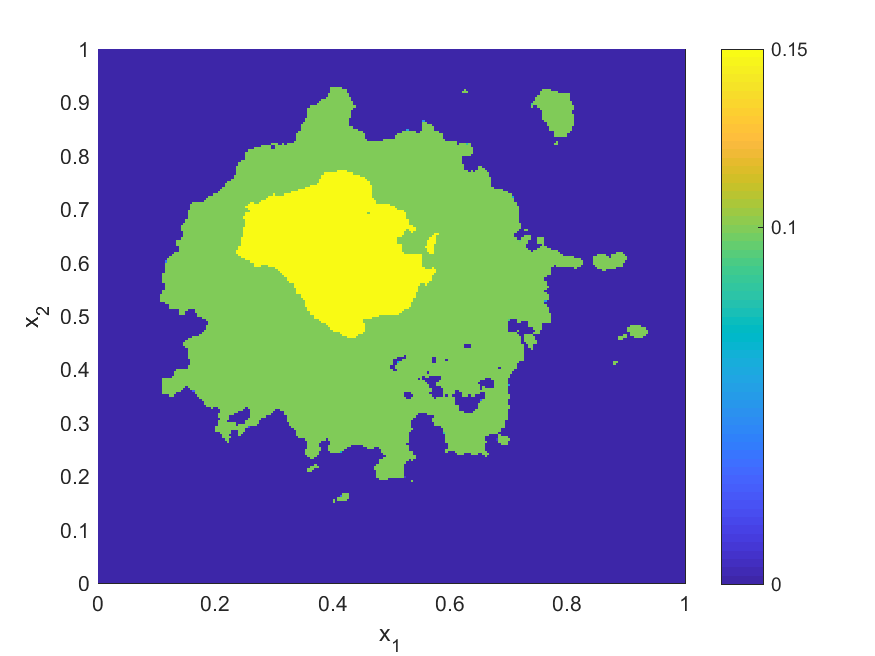}
        \caption{$u_\alpha^\delta$ for $\alpha = 10^{-5}$}
        \label{fig:recon_d05:5}
    \end{subfigure}
    \caption{True parameter $u^\dag$ and reconstructions $u_\alpha^\delta$ for $u_3=0.15$, $\delta \approx 7.59\cdot10^{-1}$, and different $\alpha$}
    \label{fig:recon_d05}
\end{figure}

\section{Conclusion}

Reconstructions in inverse problems that take on values from a given discrete admissible set can be promoted via a convex penalty that leads to a convergent regularization method. While convergence rates can be shown with respect to the usual Bregman distance, if the true parameter to be reconstructed takes on values only from the admissible set, the convergence (albeit without rates) is actually pointwise. A semismooth Newton method allows the efficient and robust computation of Tikhonov minimizers.

This work can be extended in several directions. First, \cref{fig:recon_d05} demonstrates that regularization parameters chosen according to the discrepancy principle are not optimal with respect to the visual reconstruction quality. This motivates the development of new, heuristic, parameter choice rules that are adapted to the discrete-valued, pointwise, nature of the multi-bang penalty. 
It would also be interesting to investigate whether an active set condition in the spirit of \cite{Wachsmuth:2011a,Wachsmuth:2011b} based on \eqref{eq:ass_p} can be used to obtain strong or pointwise convergence rates.
A natural further step is the extension to nonlinear parameter identification problems, making use of the results of \cite{CK:2015}.
Finally, \cref{fig:recon_d05:4,fig:recon_d05:5} suggest combining the multi-bang penalty with a total variation penalty to also promote regularity of the level lines of the reconstruction. The resulting problem is challenging both analytically and numerically, but would open up the possibility of application to electrical impedance tomography, which can be formulated as parameter identification problem for the diffusion coefficient in an elliptic equation.

\section*{Acknowledgments}
This work was supported by the German Science Fund (DFG) under grant CL 487/1-1. The authors also wish to thank Daniel Wachsmuth for several helpful remarks.

\printbibliography
\end{document}